\documentclass[a4paper,12pt]{amsart}
\usepackage{hyperref,fancyhdr,mathrsfs,amsmath,amscd,amsthm,amsfonts,latexsym,amssymb,stmaryrd}
\usepackage[all]{xy}

\newcommand{\F}{\mathscr{F}}

\def \a{\alpha}

\def \phi{\varphi}
\def \Phi{\varPhi}
\def \p{\pi}
\def \r{\rho}
\def \s{\sigma}

\def \R{\mathbb{R}}
\def \Hq{\mathbb{H}\,}
\def \C{\mathbb{C}\,}

\def\widecheckg{g^{\hspace*{-2.5pt}\vbox to 5pt{\hbox to
0pt{\LARGE$\check{}$}}}\hspace*{2pt}}

\def\widecheckl{\lambda^{\hspace*{-3.5pt}\vbox to 8pt{\hbox to
0pt{\LARGE$\check{}$}}}\hspace*{2pt}}

\begin{document}

\title{On the classification of the real vector subspaces of a quaternionic vector space}
\author{Radu Pantilie}
\email{\href{mailto:radu.pantilie@imar.ro}{radu.pantilie@imar.ro}}
\address{R.~Pantilie, Institutul de Matematic\u a ``Simion~Stoilow'' al Academiei Rom\^ane,
C.P. 1-764, 014700, Bucure\c sti, Rom\^ania}

\newtheorem{thm}{Theorem}[section]
\newtheorem{lem}[thm]{Lemma}
\newtheorem{cor}[thm]{Corollary}
\newtheorem{prop}[thm]{Proposition}

\theoremstyle{definition}

\newtheorem{defn}[thm]{Definition}
\newtheorem{rem}[thm]{Remark}
\newtheorem{exm}[thm]{Example}

\numberwithin{equation}{section}

\maketitle
\thispagestyle{empty}
\section*{Abstract}
\begin{quote}
{\footnotesize
We prove the classification of the real vector subspaces of a quaternionic vector space by using a covariant functor 
which, to any pair formed of a quaternionic vector space and a real subspace, associates a coherent sheaf over the sphere.}
\end{quote}

\section*{Introduction}

\indent
Let $X_E$ be the space of (real) vector subspaces of a vector space $E$. Then $X_E$ is a disjoint union of Grassmannians and
${\rm GL}(E)$ acts transitively on each of its components.\\
\indent
If $E$ is endowed with a linear geometric structure, corresponding to the Lie subgroup $G\subseteq{\rm GL}(E)$\,,
then it is natural to ask whether or not the action induced by $G$ on $X_E$ is still transitive on each component and, if not,
to find explicit representatives for each orbit.\\
\indent
For example, if $E$ is an Euclidean vector space and, accordingly, $G$ is the orthogonal group then the orthonormalization process
shows that $G$ acts transitively on each component of $X_E$\,.\\
\indent
Suppose, instead, that $E$ is endowed with a linear complex structure $J$\,; equivalently, $E=\C^{\!k}$ and $G={\rm GL}(k,\C)$\,.
Then for any vector subspace $U$ of $E$ we have a decomposition $U=F\times V$, where $F$ is a complex vector subspace of $E$ and $V$ is totally
real (that is, $V\cap JV=0$); obviously, the filtration $0\subseteq F\subseteq U$ is canonical. Consequently, the subspaces
$\C^{\!m}\times\R^{l}$, where $2m+l\leq 2k$\,, are representatives for each of the orbits of ${\rm GL}(k,\C)$ on $X_{\C^{\!k}}$.\\
\indent 
The corresponding decomposition for the real subspaces of a hypercomplex vector space - that is, $E=\Hq^{\!k}$ and $G={\rm GL}(k,\Hq)$ - 
was obtained in \cite{DlaRin}\,.\\ 
\indent 
By using a different method, we obtain the decomposition and the canonical filtration
for the real subspaces of a quaternionic vector space; that is, $E=\Hq^{\!k}$ and $G={\rm Sp}(1)\cdot{\rm GL}(k,\Hq)$\,. 
This involves a covariant functor from the category of pairs $(U,E)$\,, where $E$ is a quaternionic vector space
and $U\subseteq E$ is a real vector subspace (with the obvious morphisms induced by the linear quaternionic maps),
to the category of coherent sheaves on the Riemann sphere. We mention that a similar functor appeared
in \cite{Qui-QJM98} (see \cite{Wid-QJM}\,).\\
\indent
I am grateful to Stefano Marchiafava for useful discussions and comments, and to Anton Galaev for informing me about \cite{DlaRin}\,.

\section{Complex and (co-)CR vector spaces} \label{section:complex_(co-)cr_vector_spaces}

\indent
A \emph{linear complex structure} on a (real) vector space $U$ is a linear map $J:U\to U$ such that $J^2=-{\rm Id}_U$.
Then, on associating to any linear complex structure the $-{\rm i}$ eigenspace of its complexification, we obtain a (bijective)
correspondence between the space of linear complex structures on $U$ and the space of complex vector subspaces $C$ of $U^{\C}$
such that $C\oplus\overline{C}=U^{\C}$.\\
\indent
This suggests to consider the following two less restrictive conditions for a complex vector subspace $C$ of $U^{\C}$:\\
\indent
\quad1) $C\cap\overline{C}=0$\,,\\
\indent
\quad2) $C+\overline{C}=U^{\C}$.\\
\indent
Furthermore, conditions (1) and (2) are dual to each other. That is, $C\subseteq U^{\C}$ satisfies condition (1) if and only if
${\rm Ann}\,C\subseteq\bigl(U^{\C}\bigr)^*$ satisfies (2)\,, where ${\rm Ann}\,C=\bigl\{\a\in\bigl(U^{\C}\bigr)^*\,|\,\a|_C=0\bigr\}$
is \emph{the annihilator} of $C$.\\
\indent
Now, it is a standard fact that if $C\subseteq U^{\C}$ satisfies (1) then it is called a \emph{linear CR structure} on $U$.\\
\indent
Therefore a complex vector subspace $C$ of $U^{\C}$ satisfying $C+\overline{C}=U^{\C}$ is called a \emph{linear co-CR structure}
on $U$ \cite{fq}\,.\\
\indent
Thus, a complex vector subspace of $U^{\C}$ is a linear co-CR structure on $U$ if and only if its annihilator is a linear CR structure
on $U^*$.\\
\indent
A vector space endowed with a linear (co-)CR structure is a \emph{(co-)CR vector space}.\\
\indent
If $U$ is a vector subspace of a vector space $E$, endowed with a linear complex structure $J$, then $C=U^{\C}\cap E^J$
is a linear CR structure on $U$, where $E^J$ is the $-{\rm i}$ eigenspace of $J$. Moreover,
if we further assume $U+JU=E$ then $(E,J)$ is, up to complex linear isomorphisms,
the unique complex vector space, containing $U$, such that $C=U^{\C}\cap E^J$.\\
\indent
Thus, we have the following fact.

\begin{prop}[see \cite{fq}\,]
Any CR vector space corresponds to a pair $(U,E)$\,, where $(E,J)$ is a complex vector space and $U$ is a vector subspace of $E$
such that $U+JU=E$.
\end{prop}

\indent
We, also, have the following dual fact.

\begin{prop}[\,\cite{fq}\,]
Any co-CR vector space corresponds to a and pair $(V,E)$\,, where $(E,J)$ is a complex vector space and
$V$ is a vector subspace of $E$ such that $V\cap JV=0$\,.
\end{prop}
\begin{proof}
Let $(E,J)$ be a complex vector space and let $V\subseteq E$ be totally real; that is, $V\cap JV=0$\,.
Let $U=E/V$ and let $\p:E\to U$ be the projection. Then $\p(E^J)$ is a linear co-CR structure on $U$
and the proof follows quickly.
\end{proof}

\indent
Let $(E,J)$ be a complex vector space and let $U$ be a vector subspace of $E$. Then, obviously, $F=U\cap JU$ is invariant under $J$
and therefore $(F,J|_F)$ is a complex vector subspace of $(E,J)$\,. Moreover, $(F,J|_F)$ is the biggest complex vector subspace of $(E,J)$
contained by $U$. Consequently, if $V$ is a complement of $F$ in $U$ then $V$ is totally real in $E$.\\
\indent
Thus, we have a decomposition $U=F\oplus V$; moreover, the filtration $0\subseteq F\subseteq U$ is canonical.\\
\indent
As already suggested, it is useful to consider pairs $(U,E)$\,, with $E$ a complex vector space and
$U$ a vector subspace of $E$. A morphism $t:(U,E)\to(U',E')$\,, between two such pairs, is a complex linear map $t:E\to E'$
such that $t(U)\subseteq U'$. Also, there is an obvious notion of product: $(U,E)\times(U',E')=(U\times U',E\times E')$\,.

\begin{prop}[see \cite{DlaRin}\,] \label{prop:complex_decomposition}
Any pair formed of a complex vector space and a real vector subspace admits a decomposition, unique up to the order of factors,
as a (finite) product in which each factor is either $(\C,\C)$\,, $(\R,\C)$\,, or $(0,\C)$\,.
\end{prop}
\begin{proof}
Let $(E,J)$ be a complex vector space and let $U$ be a vector subspace of $E$. We have seen that $U=F\times V$, where
$F=U\cap JU$ and $V$ is a complement of $F$ in $U$\,. {}From the fact that $V\cap JV=0$ it follows that $F\cap(V+JV)=0$\,.\\
\indent
Let $E'\subseteq E$ be a complex vector subspace complementary to $F\oplus(V+JV)$\,. We, obviously, have that
$(U,E)$ is isomorphic to $(F,F)\times(V,V+JV)\times(0,E')$\,.\\
\indent
To complete the proof, just note that $(F,F)$\,, $(V,V+JV)$\,, and $(0,E')$ decompose as products in which each factor is of the
form $(\C,\C)$\,, $(\R,\C)$\,, and $(0,\C)$\,, respectively.
\end{proof}

\indent
If we apply Proposition \ref{prop:complex_decomposition} to the pair corresponding to a (co-)CR vector space then we obtain the
following facts, dual to each other:\\
\indent
\quad1) The pair corresponding to a CR vector space admits a decomposition, unique up to the order of factors,
as a product in which each factor is either $(\C,\C)$ or $(\R,\C)$\,;\\
\indent
\quad2) The pair corresponding to a co-CR vector space admits a decomposition, unique up to the order of factors,
as a product in which each factor is either $(\R,\C)$ or $(0,\C)$\,.\\
\indent
Thus, we have the following result.

\begin{cor}
Any pair formed of a complex vector space and a real vector subspace admits a decomposition as a product
of the pair corresponding to a CR vector space and the pair corresponding to a co-CR vector space.
\end{cor}

\section{Quaternionic vector spaces}

\indent
The automorphism group of the (unital) associative algebra of quaternions is ${\rm SO}(3,\R)$\,, acting trivially on $\R$ and canonically
on ${\rm Im}\Hq(=\R^3)$\,. Thus, if $E$ is a vector space then there exists a natural action of ${\rm SO}(3,\R)$
on the space of morphisms of associative algebras from $\Hq$ to ${\rm End}(E)$\,; that is, on the space
of \emph{linear hypercomplex structures} on $E$. The (nonempty) orbits of this action are the \emph{linear quaternionic structures} on~$E$.\\
\indent
A \emph{quaternionic (hypercomplex) vector space} is a vector space endowed with a linear quaternionic (hypercomplex) structure
(see \cite{AleMar-Annali96}\,,\,\cite{IMOP}\,).\\
\indent
Let $E$ be a quaternionic vector space and let $\r:\Hq\to{\rm End}(E)$ be a representative of its linear quaternionic structure.
Then, obviously, the space $Z=\rho(S^2)$ of \emph{admissible linear complex structures} on $E$ depends only of the linear quaternionic
structure of $E$. We denote by $E^J$ the $-{\rm i}$ eigenspace of $J\in Z$.\\
\indent
The linear quaternionic structure on $E$ corresponds to a linear quaternionic structure on its dual $E^*$
given by the morphism of associative algebras from $\Hq$ to ${\rm End}(E^*)$\,, which maps any $q\in \Hq$ to the transpose of $\rho(\overline{q})$\,.
Thus, any admissible linear complex structure $J$ on $E$ corresponds to the admissible linear complex structure $J^*$ which
is the opposite of the transpose of $J$; note that, $(E^*)^{J^*}$ is the annihilator of $E^J$.\\
\indent
Let $E$ and $E'$ be quaternionic vector spaces and let $Z$ and $Z'$ be the corresponding spaces of admissible linear complex structures,
respectively. A \emph{linear quaternionic map} from $E$ to $E'$ is a linear map $t:E\to E'$ such that, for some function
$T:Z\to Z'$, we have $t\circ J=T(J)\circ t$\,, for any $J\in Z$; consequently, if $t\neq0$ then $T$ is unique
and an orientation preserving isometry (see \cite{IMOP}\,).\\
\indent
The (left) $\Hq$-module structure on $\Hq^{\!k}$ determines a linear quaternionic structure on it. Moreover, for any quaternionic
vector space $E$, with $\dim E=4k$\,, there exists a linear quaternionic isomorphism from $E$ onto $\Hq^{\!k}$. The group of linear quaternionic
automorphisms of $\Hq^{\!k}$ is ${\rm Sp}(1)\cdot{\rm GL}(k,\Hq)$\,, acting on $\Hq^{\!k}$ by $\bigl(\pm(a,A),q\bigr)\mapsto aqA^{-1}$,
for any $\pm(a,A)\in{\rm Sp}(1)\cdot{\rm GL}(k,\Hq)$ and $q\in\Hq^{\!k}$ (see \cite{IMOP}\,).\\
\indent
We end this section by showing how to define the product of two quaternionic vector spaces $E$ and $E'$.
Let $T:Z\to Z'$ be an orientation preserving isometry between the spaces of admissible linear complex structures on $E$ and $E'$.\\
\indent
If $\rho:\Hq\to{\rm End}(E)$ represents the linear quaternionic structure of $E$ then
$T$ is the restriction of a unique linear map $\widetilde{T}:\r(\Hq)\to{\rm End}(E')$ such that
$\widetilde{T}\circ\rho$ determines the linear quaternionic structure on $E'$.\\
\indent
Then $q\mapsto\bigl(\rho(q),\widetilde{T}(\rho(q))\bigr)$\,, $(q\in\Hq)$\,, defines the \emph{product linear quaternionic structure}
on $E\times E'$ (with respect to $T$).\\
\indent
Note that, although the product of two quaternionic vector spaces is well-defined (that is, it doesn't depend on the particular
isometry $T$), it doesn't make the category of quaternionic vector spaces Abelian. Nevertheless, it is obvious that
the category of hypercomplex vector spaces is Abelian.

\section{Pairs formed of a quaternionic vector space\\
and a real vector subspace}

\indent
The category of quaternionic vector spaces is a full subcategory of the category whose objects are pairs
$(U,E)$\,, where $E$ is a quaternionic vector space and $U\subseteq E$ is a real vector subspace. The morphisms
between two such pairs $(U,E)$ and $(U',E')$ are the linear quaternionic maps $t:E\to E'$ such that $t(U)\subseteq U'$ (see \cite{DlaRin}\,).\\
\indent
If $U$ is a real vector subspace of a quaternionic vector space $E$ we call $({\rm Ann}\,U,E^*)$ the \emph{dual} of $(U,E)$\,.\\
\indent
We shall see that there are three basic subcategories of the category of pairs formed of a quaternionic vector space and a real vector subspace,
two of which are related to the Twistor Theory (see \cite{fq}\,).

\begin{defn} \label{defn:(co-)cr_q}
Let $E$ be a quaternionic vector space and let $Z$ be its space of admissible linear complex structures.\\
\indent
If $\iota:U\to E$ is an injective linear map then $(E,\iota)$ is a \emph{linear CR quaternionic structure} on $U$ if
${\rm im}\,\iota+J({\rm im}\,\iota)=E$, for any $J\in Z$.\\
\indent
A \emph{CR quaternionic vector space} is a vector space endowed with a linear CR quaternionic structure.
\end{defn}

\indent
By duality, we obtain the notion of \emph{co-CR quaternionic vector space}.\\
\indent
To any co-CR quaternionic vector space $(U,E,\r)$ we associate the pair $({\rm ker}\r,E)$\,. Thus, the category of co-CR quaternionic
vector spaces is a full subcategory of the category of pairs formed of a quaternionic vector space and a real vector subspace;
by duality, the latter, also, includes the category of CR quaternionic vector spaces.\\
\indent
See \cite{fq} for further information on (co-)CR quaternionic vector spaces.

\begin{rem}
1) Let $U$ be a real vector subspace of a quaternionic vector space $E$. Then $(U,E)$ is given by a CR quaternionic vector space
if and only if its dual is given by a co-CR quaternionic vector space.\\
\indent
2) Any quaternionic vector space $E$ is both CR and co-CR quaternionic. When we consider $E$ a CR quaternionic vector space
the associated pair is $(E,E)$\,, whilst when we consider $E$ a co-CR quaternionic vector space the associated pair is $(0,E)$\,.
\end{rem}

\indent
We shall construct a covariant functor from the category of pairs, formed of a quaternionic vector space and a real vector subspace,
to the category of coherent analytic sheaves, over the sphere, endowed with a conjugation covering the antipodal map
(see \cite{GunRos} for the basic properties of coherent analytic sheaves and \cite{Qui-QJM98} for coherent analytic sheaves, over the sphere,
endowed with a conjugation covering the antipodal map, briefly called `$\s$-sheaves').\\
\indent
For this, firstly, note that if $E$ is a quaternionic vector space, with $Z\,(=S^2)$ the space of admissible
linear complex structures, then $E^{0,1}=\bigcup_{J\in Z} \{J\}\times E^J$ is a holomorphic vector subbundle of $Z\times E^{\C}$.
Now, if $U\subseteq E$ is a real vector subspace then the projection $E\to E/U$ induces, by restriction, a morphism of
holomorphic vector bundles $E^{0,1}\to Z\times(E/U)^{\C}$. Let $\mathcal{U}_-$ and $\mathcal{U}_+$ be the kernel and cokernel,
respectively, of this morphism of holomorphic vector bundles.

\begin{defn}
We call $\mathcal{U}=\mathcal{U}_-\oplus\mathcal{U}_+$ \emph{the (coherent analytic) sheaf of $(U,E)$\,}.
\end{defn}

\indent
The proof of the following proposition is straightforward.

\begin{prop} \label{prop:assoc_sheaf_first_props}
The association $(U,E)\mapsto\mathcal{U}$ defines a covariant functor $\F$ from the category of pairs, formed of a quaternionic vector space
$E$ and a real vector subspace $U\subseteq E$, to the category of coherent sheaves, on the sphere, endowed with a conjugation covering
the antipodal map. Furthermore, $\F$ has the following properties:\\
\indent
{\rm (i)} For any morphism $t:(U,E)\to(U',E')$\,, we have that $\F(t)$ maps $\F(U,E)_{\pm}$ to $\F(U',E')_{\pm}$\,.\\
\indent
{\rm (ii)} If $(U,E)$ is given by a (co-)CR quaternionic vector space then $\F(U,E)$ is its holomorphic vector bundle.
\end{prop}

\indent
With the same notations as in Proposition \ref{prop:assoc_sheaf_first_props}\,, if\/ $\mathcal{U}=\mathcal{U}_+$
then $E/U$ is the space of (global) sections of\/ $\mathcal{U}$ intertwining the antipodal map and the conjugation.\\
\indent
Here are the basic examples of pairs whose sheaves are torsion free (cf.\ \cite{DlaRin}\,, \cite{Qui-QJM98}\,; see, also, \cite{Wid-QJM}\,, \cite{fq}\,).

\begin{exm} \label{exm:U_k}
Let $q_1\,,\ldots,q_{k+1}\in S^2$, $(k\geq1)$\,, be such that $q_i\neq\pm q_j$\,, if $i\neq j$\,.
For $j=1,\ldots,k$ let $e_j=(\underbrace{0,\ldots,0}_{j-1},q_j,q_{j+1},\underbrace{0,\ldots,0}_{k-j})$\,.
Denote $U_0=\R$ and, for $k\geq1$\,, let $U_k=\R^{k+1}+\R e_1+\ldots+\R e_k$\,.\\
\indent
Then the sheaf of $(U_k,\Hq^{\!k+1})$ is $\mathcal{O}(2k+2)$\,, for any $k\in\mathbb{N}$.
Note that, the projection $\Hq^{\!k+1}\to\Hq^{\!k+1}/U_k$ defines a co-CR quaternionic vector space
and the sheaf of the dual of $(U_k,\Hq^{\!k+1})$ is $\mathcal{O}(-2k-2)$\,, for any $k\in\mathbb{N}$.
\end{exm}

\begin{exm} \label{exm:V_k}
Let $V_0=\{0\}$ and, for $k\geq1$\,, let $V_k$ be the vector subspace of $\Hq^{\!2k+1}$ formed of all vectors of the form
$$(z_1\,,\overline{z_1}+z_2\,{\rm j}\,, z_3-\overline{z_2}\,{\rm j}\,,\ldots, \overline{z_{2k-1}}+z_{2k}\,{\rm j}\,,-\overline{z_{2k}}\,{\rm j})\;,$$
where $z_1\,,\ldots,z_{2k}$ are complex numbers.\\
\indent
Then the sheaf of $(V_k,\Hq^{\!2k+1})$ is $2\mathcal{O}(2k+1)$\,, for any $k\in\mathbb{N}$.
Note that, the projection $\Hq^{\!2k+1}\to\Hq^{\!2k+1}/V_k$ defines a co-CR quaternionic vector space
and the sheaf of the dual of $(V_k,\Hq^{\!2k+1})$ is $2\mathcal{O}(-2k-1)$\,, for any $k\in\mathbb{N}$.
\end{exm} 

\indent 
The next class of pairs is taken from \cite{DlaRin}\,. 

\begin{exm} \label{exm:W_kq}
For $k\geq1$ and $q\in S^2$\,, let $W_{k,q}$ be the real vector subspace of $\Hq^{\!k}$ formed of all vectors of the form 
\begin{equation} \label{e:DlaRin}
(a_1+b_1q+b_2{\rm i}\,, a_2+b_2q+b_3{\rm i}\,,\ldots,a_{k-1}+b_{k-1}q+b_k{\rm i}\,,a_k+b_kq)\;, 
\end{equation} 
where $a_1,b_1,\ldots,a_k,b_k$ are real numbers, and we have assumed $q\neq\pm{\rm i}$\,; 
if $q=\pm{\rm i}$ then we replace ${\rm i}$ by ${\rm j}$ in \eqref{e:DlaRin}\,.\\ 
\indent 
Then for any $p\in S^2\setminus\{\pm q\}$ we have $W_{k,q}\cap p\,W_{k,q}=0$\,, whilst $W_{k,q}\cap q\,W_{k,q}$ has dimension two. 
Together with \cite[Proposition 3.1]{Qui-QJM98} this implies that the sheaf of $(W_{k,q},\Hq^{\!k})$ is the indecomposable torsion sheaf 
with conjugation, supported at $\pm q$\,, and of Chern number $2k$\,. 
\end{exm}

\section{The main results}

\indent
Now, we can prove the following:

\begin{thm}[cf.\ \cite{DlaRin}\,] \label{thm:qv_main}
Any pair formed of a quaternionic vector space and a real vector subspace admits a decomposition, unique up to the order of factors,
as a (finite) product in which each factor is given by one of the Examples \ref{exm:U_k}\,, \ref{exm:V_k} or \ref{exm:W_kq}\,,
or is the dual of one of the Examples \ref{exm:U_k} or \ref{exm:V_k}\,.
\end{thm}
\begin{proof} 
Let $\mathcal{U}$ be the sheaf of $(U,E)$\,. By the dual of \cite[Proposition 4.7]{fq}\,, we have that $\mathcal{U}_-$ is the holomorphic vector bundle
of a CR quaternionic vector space $(U_-,E_-)$\,. Furthermore, from the diagram of the proof of \cite[Theorem~4.8]{Qui-QJM98}
(adapted to the case of sheaves with conjugations) we obtain that there exists an injective morphism $t:(U_-,E_-)\to(U,E)$
which induces an injective linear map $E_-/U_-\to E/U$; equivalently, $U_-=E_-\cap t^{-1}(U)$\,. Therefore $t$ admits a
cokernel $(U_+,E_+)$ whose sheaf is, obviously, $\mathcal{U}_+$\,.\\
\indent
Thus, we may assume $\mathcal{U}=\mathcal{U}_+$ and, consequently, we have an exact sequence 
\begin{equation} \label{e:E/U}
0\longrightarrow E^{0,1}\longrightarrow Z\times(E/U)^{\C}\longrightarrow\mathcal{U}\longrightarrow 0\;. 
\end{equation} 
\indent 
Then the cohomology exact sequence of \eqref{e:E/U} gives a canonical isomorphism (which intertwines the conjugations) 
$(E/U)^{\C}=H^0(Z,\mathcal{U})$\,.\\ 
\indent 
Furthermore, the morphism $(0,E)\to(U,E)$ determines a surjective sheaf morphism 
$\mathcal{E}\to\mathcal{U}$ whose kernel is $Z\times U^{\C}$ (with the corresponding morphism to $\mathcal{E}$ given by the 
inclusion $Z\times U^{\C}\to Z\times E^{\C}$ followed by the projection $Z\times E^{\C}\to\mathcal{E}$). Thus, we, also, have 
\begin{equation} \label{e:U+} 
0\longrightarrow Z\times U^{\C}\longrightarrow\mathcal{E}\longrightarrow\mathcal{U}\longrightarrow 0\;. 
\end{equation} 
\indent 
The cohomology exact sequence of \eqref{e:U+}\,, together with the isomorphisms $(E/U)^{\C}=H^0(Z,\mathcal{U})$ and $E^{\C}=H^0(Z,\mathcal{E})$\,, 
show that the inclusion $U\to E$ is determined by $\mathcal{U}$.\\ 
\indent 
Now, tensorising \eqref{e:U+} with the tautological line bundle over $Z\,(=\C\!P^1)$ and by using \cite[Proposition 3.1]{Qui-QJM98} we deduce that 
in the Birkhoff--Grothendieck decomposition of $\mathcal{U}$ there are no trivial terms. Together with Examples \ref{exm:U_k}\,, \ref{exm:V_k}\,,  
and \ref{exm:W_kq}\,, this completes the proof. 
\end{proof}

\indent
Let $(U,E)$ be a pair formed of a quaternionic vector space and a real vector subspace.\\
\indent
Then $(U,E)$ is a \emph{torsion pair} if it corresponds to a torsion sheaf; equivalently, $(U,E)$ is a product of pairs
as in Example \ref{exm:W_kq}\,.\\
\indent
The pair $(U,E)$ is \emph{torsion free} if its sheaf is torsion free; equivalently, it is a holomorphic vector bundle.

\begin{cor} \label{cor:vq_second_main_cor}
{\rm (i)} Let $(U,E)$ and $(U',E')$ be pairs formed of a quaternionic vector space and a real vector subspace. Suppose that either
$(U,E)$ or $(U',E')$ are torsion free. Then $(U\times U',E\times E')$ doesn't depend of the particular isometry used to define $E\times E'$.\\
\indent
{\rm (ii)} Any pair $(U,E)$ formed of a quaternionic vector space and a real vector subspace decomposes uniquely as the product
of a torsion pair and (the pairs given by) a CR quaternionic vector space and a co-CR quaternionic vector space; moreover, the filtration
$(0,0)\subseteq(U_-,E_-)\subseteq(U_-,E_-)\times(U_t,E_t)\subseteq(U,E)$
is canonical, where $(U_t,E_t)$ is the torsion pair and $(U_-,E_-)$ is the CR quaternionic vector space.
\end{cor}
\begin{proof}
If\/ $\mathcal{U}$ is a holomorphic vector bundle over $S^2$ and $T:S^2\to S^2$ is a holomorphic diffeomorphism then $T^{-1}(\mathcal{U})$
is isomorphic to $\mathcal{U}$ and, furthermore, the same holds for bundles, endowed with a conjugation covering the antipodal map,
and their pull-backs through orientation preserving isometries. Assertion (i) follows quickly.\\
\indent
Assertion (ii) follows from (i) and the proof of Theorem \ref{thm:qv_main}\,.
\end{proof}

\indent
Finally, note that the `augmented (strengthen) $\Hq$-modules' of \cite{Joy-QJM98} (\,\cite{Qui-QJM98}\,) are just pairs whose decompositions
contain no terms of the form $(\Hq,\Hq)$ (\,$(0,\Hq)$\,); equivalently, in the decompositions of their sheaves there are no terms
of Chern number $-1$ (\,$1$\,).

\end{document}